\documentclass[11pt]{preprint}
\usepackage{times}
\usepackage{microtype}
\usepackage{mhenvs}
\usepackage{amsmath}
\usepackage{mhequ}
\usepackage{Symbols}
\usepackage{ScriptFonts}

\def\E{{\symb E}}
\def\P{{\symb P}}
\def\cD{{\cscr D}}
\def\cA{{\cscr A}}
\def\cM{{\cscr M}}
\def\cV{{\cscr V}}
\def\cC{{\cscr C}}

\def\span{\mathop{\mathrm{span}}}

\newtheorem{assumption}[lemma]{Assumption}

\definecolor{darkred}{rgb}{0.9,0.1,0.1}

\begin{document}

\title{On Malliavin's proof of H\"ormander's theorem}
\author{Martin Hairer}
\institute{Mathematics Department, University of Warwick\\ \email{M.Hairer@Warwick.ac.uk}}
\maketitle

\begin{abstract}
The aim of this note is to provide a short and self-contained proof of H\"ormander's
theorem about the smoothness of transition probabilities for a diffusion under 
H\"ormander's ``brackets condition''. While both the result and the technique of
proof are well-known, the exposition given here is novel in two aspects.
First, we introduce Malliavin calculus in an ``intuitive'' way, without using Wiener's
chaos decomposition. While this may make it difficult to prove some of the standard
results in Malliavin calculus (boundedness of the derivative operator in $L^p$ spaces
for example), we are able to bypass these and to replace them by weaker results that
are still sufficient for our purpose. Second, we introduce a notion of ``almost implication''
and ``almost truth'' (somewhat similar to what is done in fuzzy logic) which allows,
once the foundations of Malliavin calculus are laid out, to give a very short and streamlined 
proof of H\"ormader's theorem that focuses on the main ideas without clouding it by
technical details.
\\[3mm]
\noindent \textit{Dedicated to the memory of Paul Malliavin.}
\end{abstract}

\section{Introduction}

One of the main tools in many results on the convergence to equilibrium of Markov processes is the 
presence of some form of ``smoothing'' for the semigroup. For example, if a Markov operator
$\CP$ over a Polish space $\CX$ possesses the strong Feller property (namely it maps $\CB_b(\CX)$, the space
of bounded measurable functions into $\CC_b(\CX)$, the space of 
bounded continuous functions), then one can conclude that any two ergodic invariant measures for $\CP$
must either coincide or have disjoint topological supports. Since the latter can often been ruled out by some form of
controllability argument, we see how the strong Feller property is the basis for many proofs of ergodicity.

It is then desirable to have criteria that are as simple to formulate as possible and that ensure 
that the Markov semigroup associated to a given Markov process has some smoothing property.
One of the most natural classes of Markov processes are given by diffusion processes and this will be
the object of study in this note. Our main object of study is a stochastic differential equation of the form
\begin{equ}[e:SDE]
dx = V_0(x)\,dt + \sum_{i=1}^m V_i(x)\circ dW_i\;,
\end{equ}
where the $V_i$'s are smooth vector fields on $\R^n$ and the $W_i$'s are independent standard Wiener processes. 
In order to keep all arguments as straightforward as possible, we will assume throughout this note that
these vector fields assume the coercivity
assumptions necessary so that the solution flow to \eref{e:SDE}
 is smooth with respect to its initial condition and that all of its derivatives have moments of all orders.
 This is satisfied for example if the $V_i$'s are $\CC^\infty$ with bounded derivatives of all orders.
 
\begin{remark}
We wrote \eref{e:SDE} as a Stratonowich equation on purpose. This is for two reasons: at a pragmatic level,
this is the ``correct'' formulation which allows to give a clean statement of H\"ormander's theorem (see Definition~\ref{def:HorBrack} below).
At the intuitive level, the question of smoothness of transition probabilities is related to that of
the extent of their support. The Stroock-Varadhan support theorem \cite{SVSupp} characterises this as consisting precisely
of the closure of the set of points that can be reached if the Wiener processes $W_i$ in \eref{e:SDE} are replaced by arbitrary smooth
control functions. This would not be true in general for the It\^o formulation. 
\end{remark}
 
It is well-known that if the equation \eref{e:SDE} is elliptic namely if, for every point $x \in \R^n$, the linear
span of $\{V_i(x)\}_{i=1}^m$ is all of $\R^n$, then the law of the solution to \eref{e:SDE} has a smooth density with
respect to Lebesgue measure. Furthermore, the corresponding Markov semigroup $\CP_t$ defined by
\begin{equ}
\CP_t \phi(x_0) = \E_{x_0} \phi(x_t)\;,
\end{equ}
is so that $\CP_t \phi$ is smooth, even if $\phi$ is only bounded measurable. (Think of the solution to the heat
equation, which corresponding to the simplest case where $V_0= 0$ and the $V_i$ form an orthonormal basis of $\R^n$.)
In practice however, one would like to obtain a criterion that also applies to some equations where the ellipticity assumption fails.
For example, a very well-studied model of equilibrium statistical mechanics is given by the Langevin equation:
\begin{equ}
dq = p\,dt \;,\qquad dp = -\nabla V(q)\,dt - p\,dt + \sqrt{2T}\,dW(t)\;,
\end{equ}
where $T>0$ should be interpreted as a temperature, $V\colon \R^n \to \R_+$ is a sufficiently coercive potential function,
and $W$ is an $n$-dimensional Wiener process. Since solutions to this equation take values in $\R^{2n}$ (both $p$ and $q$
are $n$-dimensional), this is definitely not an elliptic equation. At an intuitive level however, one would expect it to have some
smoothing properties: smoothing reflects the spreading of our uncertainty about the position of the solution and the 
uncertainty on $p$ due to the presence of the noise terms gets instantly transmitted to $q$ via the equation $dq = p\,dt$.

In a seminal paper \cite{Hor67AM}, H\"ormander was the first to formulate the ``correct'' non-degeneracy condition 
ensuring that solutions to \eref{e:SDE} have a smoothing effect. 
To describe this non-degeneracy condition, recall that the Lie bracket $[U,V]$ between two vector fields $U$ and $V$ 
on $\R^n$ is the vector field defined by
\begin{equ}{}
[U,V](x) = DV(x)\,U(x) - DU(x)\,V(x)\;,
\end{equ}
where we denote by $DU$ the derivative matrix given by $(DU)_{ij} = \d_j U_i$.
This notation is consistent with the usual notation for the commutator between two linear operators since,
if we denote by $A_U$ the first-order differential operator acting on smooth functions $f$ by
$A_U f(x) = \scal{U(x),\nabla f(x)}$, then we have the identity $A_{[U,V]} = [A_U,A_V]$.

With this notation at hand, we give the following definition:

\begin{definition}\label{def:HorBrack}
Given an SDE \eref{e:SDE}, define a collection of vector fields $\cV_k$ by
\begin{equ}
\cV_0 = \{V_i\,:\, i > 0\}\;,\quad \cV_{k+1} = \cV_k \cup \{[U,V_j]\,:\, U\in \cV_k\;\&\; j \ge 0\}\;.
\end{equ}
We also define the vector spaces $\cV_k(x) = \span\{V(x)\,:\, V \in \cV_k\}$. We say that \eref{e:SDE}
 satisfies the \textit{parabolic H\"ormander condition} if $\bigcup_{k \ge 1} \cV_k(x) = \R^n$ for every $x \in \R^n$.
\end{definition}

With these notations, H\"ormander's theorem can be formulated as

\begin{theorem}\label{theo:Hormander}
Consider \eref{e:SDE} and assume that all vector fields have bounded derivatives of all orders. If it satisfies the parabolic 
H\"ormander condition, then its solutions admit a smooth density with respect to Lebesgue measure and the corresponding Markov
semigroup maps bounded functions into smooth functions. 
\end{theorem}

H\"ormander's original proof was formulated in terms of second-order differential operators and was
purely analytical in nature. Since the main motivation on the other hand was probabilistic and since, as we will see below,
H\"ormander's condition can be understood at the level of properties of the trajectories of \eref{e:SDE}, a more
stochastic proof involving the original stochastic differential equation was sought for. The breakthrough came
with Malliavin's seminal work \cite{Mal78SCO}, where he laid the foundations of what is now known as the 
``Malliavin calculus'', a differential calculus in Wiener space and used it to give a probabilistic proof of 
H\"ormander's theorem. This new approach proved to be extremely successful and soon a number of authors
studied variants and simplifications of the original proof \cite{Bismut1,Bismut2,KSAMI,KSAMII,KSAMIII,Nor86SMC}.
Even now, more than three decades after Malliavin's original work, his techniques prove to be sufficiently flexible
to obtain related results for a number of extensions of the original problem, including
for example SDEs with jumps \cite{MR1932281,MR2307057,MR2513114,Takeuchi}, 
infinite-dimensional systems \cite{MR953905,MR2152244,MR2257860,MR2259251,ErgodicBig}, and SDEs driven by Gaussian
processes other than Brownian motion \cite{MR2322701,CassFriz,HairPil}.

A complete rigorous proof of Theorem~\ref{theo:Hormander} goes somewhat beyond the scope of these notes.
However, we hope to be able to give a convincing argument showing why this result is true and what are the
main steps involved in its probabilistic proof. The aim in writing these notes was to be sufficiently self-contained so
that a strong PhD student interested in stochastic analysis would be able to fill in the missing gaps without
requiring additional ideas.
The interested reader can find the technical details required to make the proof rigorous
in \cite{Mal78SCO,KSAMI,KSAMII,KSAMIII,Nor86SMC,Nua95}. H\"ormander's original, completely different, 
proof using fractional integrations can be found
in \cite{Hor67AM}. A yet completely different functional-analytic proof using the theory of pseudo-differential operators
was developed by Kohn in \cite{Koh77} and can also be found in \cite{Hor85} or, in a slightly
different context, in the recent book \cite{MR2130405}.

The remainder of these notes is organised as follows. First, in Section~\ref{sec:heuristics} below, we will show why
it is natural that the iterated Lie brackets appear in H\"ormander's condition. Then, in Section~\ref{sec:Malliavin}, 
we will give
an introduction to Malliavin calculus, including in particular its integration by parts formula in Wiener space.
Finally, in Section~\ref{sec:Hormander}, we apply these tools to the particular case of smooth diffusion
processes in order to give a probabilistic proof of H\"ormander's theorem.

\subsection*{Acknowledgements}

{\small
These notes were part of a minicourse given at the University of Warwick
in July 2010; thanks are due to the organisers for this pleasant event.
I would also like to thank Michael Scheutzow and Hendrik Weber who carefully read a previous version of
the manuscript and pointed out several misprints. Remaining mistakes were probably added later on!
Finally, I would like to acknowledge financial support provided by the EPSRC through grants EP/E002269/1 and EP/D071593/1, 
the Royal Society through a Wolfson Research Merit Award, and the Leverhulme Trust through a Philip Leverhulme prize. 
}

\section{Why is it the correct condition?}
\label{sec:heuristics}

At first sight, the condition given in Definition~\ref{def:HorBrack} might seem a bit strange. Indeed,
the vector field $V_0$ is treated differently from all the others: it appears in the recursive definition of
 the $\cV_{k}$, but not in $\cV_0$. This can be understood in the following way: consider trajectories of
\eref{e:SDE} as curves in space-time. By the Stroock-Varadhan support theorem \cite{SVSupp}, the law of the solution to
\eref{e:SDE} on pathspace is supported by the closure of those smooth curves that, at every point $(x,t)$, 
are tangent to the hyperplane spanned
 by $\{\hat V_0, \ldots, \hat V_m\}$, where we set
\begin{equ}
\hat V_0(x,t) = \begin{pmatrix}V_0(x) \\ 1\end{pmatrix}\;,\qquad
\hat V_j(x,t) = \begin{pmatrix}V_j(x) \\ 0\end{pmatrix}\;,\quad j=1,\ldots,m\;.
\end{equ}
With this notation at hand, we could define $\hat \cV_k$ as in Definition~\ref{def:HorBrack}, but with
$\hat\cV_0 = \{\hat V_0,\ldots,\hat V_m\}$. Then, it is easy to check that H\"ormander's condition is
equivalent to the condition that $\bigcup_{k \ge 1} \hat \cV_k = \R^{n+1}$ for every $(x,t) \in \R^{n+1}$.

This condition however has a simple geometric interpretation. For a smooth manifold $\CM$, recall
that $E \subset T\CM$ is a smooth subbundle of dimension $d$ if $E_x \subset T_x\CM$ is a vector
space of dimension $d$ at every $x \in \CM$ and if the dependency $x \mapsto E_x$ is smooth.
(Locally, $E_x$ is the linear span of finitely many smooth vector fields on $\CM$.)
A subbundle is called \textit{integrable} if, whenever $U,V$ are vector fields on $\CM$ taking
values in $E$, their Lie bracket $[U,V]$ also takes values in $E$.

With these definitions at hand, recall the well-known
Frobenius integrability theorem from differential geometry:

\begin{theorem}\label{theo:Frobenius}
Let $\CM$ be a smooth $n$-dimensional manifold and let $E \subset T\CM$ be a smooth vector bundle
of dimension $d < n$. Then $E$ is integrable if and only if there (locally) exists a smooth foliation of $\CM$
into leaves of dimension $d$ such that, for every $x\in \CM$, the tangent space of the leaf passing through 
$x$ is given by $E_x$.
\end{theorem}

In view of this result, H\"ormander's condition is not surprising. Indeed, if we define
$E_{(x,t)} = \bigcup_{k \ge 0} \hat \cV_k(x,t)$, then this gives us a subbundle of $\R^{n+1}$
which is integrable by construction of the $\hat \cV_k$. Note that the dimension of $E_{(x,t)}$ could
in principle depend on $(x,t)$, but since the dimension is a lower semicontinuous function, it will
take its maximal value on an open set. If, on some open set, this maximal value is less than $n+1$, 
then Theorem~\ref{theo:Frobenius}
tells us that, there exists a submanifold (with boundary) $\bar \CM \subset \CM$ of dimension strictly less than $n$
such that $T_{(y,s)} \bar \CM = E_{(y,s)}$ for every $(y,s) \in \bar \CM$. In particular, all 
the curves appearing in the Stroock-Varadhan support theorem and supporting the law of the solution to \eref{e:SDE}
must lie in $\bar \CM$ until they reach its boundary. As a consequence, since $\bar \CM$ is always
transverse to the sections with constant $t$, the solutions at time $t$ will, with positive probability, lie in
a submanifold of $\CM$ of strictly positive codimension. This immediately implies that the transition probabilities
cannot be continuous with respect to Lebesgue measure.

To summarise, if H\"ormander's condition fails on an open set, then transition probabilities cannot have a
density with respect to Lebesgue measure, thus showing that H\"ormander's condition is ``almost necessary''
for the existence of densities. The hard part of course is to show that it is a sufficient condition.
Intuitively, the reason is that H\"ormander's condition allows the solution to \eref{e:SDE} to ``move in all directions''.
Why this is so can be seen from the following interpretation of the Lie brackets.
Set
\begin{equ}
u_n(t) = {1\over n} \cos(n^2 t)\;,\qquad v_n(t) = {1\over n} \sin(n^2 t)\;,
\end{equ}
and consider the solution to 
\begin{equ}[e:area]
\dot x = U(x)\,\dot u_n(t) + V(x)\,\dot v_n(t)\;.
\end{equ}
We claim that, as $n \to \infty$, this converges to the solution to
\begin{equ}[e:LieLimit]
\dot u = {1\over 2}[U,V](x)\;.
\end{equ}
This can be seen as follows. If we integrate \eref{e:area} over a short time interval, we 
have the first order approximation
\begin{equ}
x(h) \approx x^{(1)}(h) \eqdef x_0 + U(x_0) u_n(h) + V(x_0) v_n(h)\;,
\end{equ}
which simply converges to $x_0$ as $n \to \infty$. To second order, however, we have
\begin{equs}
x(h) &\approx x_0 + \int_0^h \bigl(U(x^{(1)})\,\dot u_n + V(x^{(1)})\,\dot v_n\bigr)\,dt \\
&\approx x^{(1)}(h) + \int_0^h \bigl(DU(x_0) \,\dot u_n + DV(x_0)\,\dot v_n\bigr) \bigl(U(x_0) u_n + V(x_0) v_n\bigr)\,dt \\
&\approx x_0 + \int_0^h \bigl(DU(x_0) V(x_0) v_n \dot u_n + DV(x_0) U(x_0) u_n \dot v_n\bigr)\,dt\;.
\end{equs}
Here, we used the fact that the integral of $u_n \dot u_n$ (and similarly for $v_n \dot v_n$) is given by ${1\over 2} u_n^2$ and 
therefore converges to $0$ as $n \to 0$. Note now that over a period, $v_n(t) \dot u_n(t)$ averages
to $-{1\over 2}$ and $u_n(t) \dot v_n(t)$ averages to ${1\over 2}$, thus showing that one does indeed
obtain \eref{e:LieLimit} in the limit. This reasoning shows that, by combining motions in the directions
$U$ and $V$, it is possible to approximate, to within arbitrary accuracy, motion in the direction
$[U,V]$. 

A similar reasoning shows that if we consider
\begin{equ}
\dot x = U(x) + V(x)\,\dot v_n(t)\;,
\end{equ}
then, to lowest order in $1/n$, we obtain that as $n \to \infty$, $x$ follows
\begin{equ}
\dot x \approx U(x) + {1\over 2n}[U,V](x)\;.
\end{equ}
Combining these interpretations of the meaning of Lie brackets with the Stroock-Varadhan support
theorem, it suggests that, if H\"ormander's condition holds, then the support of the law of $x_t$ will
contain an open set around the solution at time $t$ to the deterministic system
\begin{equ}
\dot x = V_0(x)\;,\qquad x(0) = x_0\;.
\end{equ}
This should at least render it plausible that under these conditions, the law of $x_t$ has a density with
respect to Lebesgue measure. The aim of this note is to demonstrate how to turn this heuristic 
into a mathematical theorem with, hopefully,  a minimum amount of effort.

\begin{remark}
While H\"ormander's condition implies that the control system associated to \eref{e:SDE}
reaches an open set around the solution to the deterministic equation $\dot x = V_0(x)$,
it does \textit{not} imply in general that it can reach an open set around $x_0$. In particular,
it is \textit{not} true that the parabolic H\"ormander condition implies that \eref{e:SDE} can reach
every open set. A standard counterexample is given by
\begin{equ}
dx = -\sin(x)\,dt + \cos(x)\circ dW(t)\;,\qquad x_0 = 0\;,
\end{equ}
which satisfies H\"ormander's condition but can never exit the interval $[-\pi/2,\pi/2]$.
\end{remark}

\section{An Introduction to Malliavin calculus}
\label{sec:Malliavin}

In this section, we collect a number of tools that will be needed in the proof.
The main tool is the integration by parts formula from Malliavin calculus, as well
of course as Malliavin calculus itself. 

The main tool in the proof is the Malliavin calculus with its integration by part formula in Wiener space,
which was developed precisely in order to provide a probabilistic proof of Theorem~\ref{theo:Hormander}.
It essentially relies on the fact that the image of a Gaussian measure under a smooth submersion
that is  sufficiently integrable possesses a smooth density with respect to Lebesgue measure.
This can be shown in the following way. First, one observes the following fact:

\begin{lemma}\label{lem:IBP}
Let $\mu$
be a probability measure on $\R^n$ such that the bound
\begin{equ}
\Bigl|\int_{\R^n} D^{(k)}G(x)\,\mu(dx) \Bigr| \le C_k \|G\|_{\infty}\;,
\end{equ}
holds for every smooth bounded function $G$ and every $k \ge 1$. Then $\mu$ has a smooth density with 
respect to Lebesgue measure. \qed
\end{lemma}

\begin{proof}
Let $s > n/2$ so that $H^s \subset \CC_b$ by Sobolev embedding.
By duality, the assumption then implies that every distributional derivative of $\mu$ belongs to the Sobolev space $H^{-s}$,
so that $\mu$ belongs to $H^\ell$ for every $\ell \in \R$.
The result then follows from the fact that $H^\ell \subset \CC^k$ as soon as $\ell > k + {n\over2}$.
\end{proof}

Consider now a sequence of $N$ independent Gaussian random 
variables $\delta w_k$ with variances $\delta t_k$ for $k \in \{1,\ldots, N\}$, as well as a smooth map
$X \colon \R^N \to \R^n$. We also denote by $w$ the collection $\{\delta w_k\}_{k \ge 1}$ and we
define the $n \times n$ matrix-valued map
\begin{equ}[e:defMall]
\cM_{ij}(w) = \sum_k \d_k X_i(w) \d_k X_j(w)\,\delta t_k\;,
\end{equ}
where we use $\d_k$ as a shorthand for the partial derivative with respect to the variable $\delta w_k$.
With this notation, $X$ being a submersion is equivalent to $\cM(w)$ being invertible for every $w$.

Before we proceed, let us introduce additional notation, which hints at the fact that one would really
like to interpret the $\delta w_k$ as the increments of  a Wiener process of an interval of length $\delta t_k$.
When considering a family $\{F_k\}_{k=1}^N$ of maps from $\R^N \to \R^n$, we identify it
with a continuous family $\{F_t\}_{t \ge 0}$, where 
\begin{equ}[e:identif]
F_t \eqdef F_k \;,\quad t \in [t_k, t_{k+1})\;,\qquad t_k \eqdef \sum_{\ell \le k} \delta t_\ell\;.
\end{equ}
Note that with this convention, we have $t_0 = 0$, $t_1 = \delta t_1$, etc. This is of course an abuse of
notation since $F_t$ is not equal to $F_k$ for $t = k$, but we hope that it will always be clear from the context
whether the index is a discrete or a continuous variable. We also set $F_t = 0$ for $t \ge t_N$.
With this notation, we have the natural identity
\begin{equ}
\int F_t \,dt = \sum_{k=1}^N F_k\, \delta t_k\;.
\end{equ}
Furthermore, given a smooth map $G\colon \R^N \to \R$, we will from now on denote by $\cD_t G$ the
family of maps such that $\cD_t G = \d_k G$ for $t \in  [t_k, t_{k+1})$, so that \eref{e:defMall} can be rewritten
as
\begin{equ}
\cM_{ij}(w) = \int \cD_t X_i(w) \cD_t X_j(w)\,dt\;.
\end{equ}
The quantity $\cD_t G$ is called the \textit{Malliavin derivative} of the random variable $G$.

The main feature of the Malliavin derivative operator $\cD_t$ suggesting that one expects it to be well-posed in
the limit $N \to \infty$ is that it was set up in such a way that it is invariant under refinement of the mesh $\{\delta t_k\}$ in the following way.
For every $k$, set
$\delta w_k = \delta w_k^- + \delta w_k^+$, where $\delta w_k^\pm$ are independent Gaussians
with variances $\delta t_k^\pm$ with $\delta t_k^- + \delta t_k^+ = \delta t_k$ and then identify maps
$G \colon \R^N \to \R$ with a map $\bar G \colon \R^{2N} \to \R$ by
\begin{equ}
\bar G(\delta w_1^\pm,\ldots,\delta w_N^\pm) = G(\delta w_1^- + \delta w_1^+,\ldots,\delta w_N^- + \delta w_N^+)\;.
\end{equ}
Then, for every $t \ge 0$, $\cD_t \bar G$ is precisely the map identified with $\cD_t G$.

With all of these notations at hand, we then have the following result:

\begin{theorem}\label{theo:smooth}
Let $X\colon \R^N \to \R$ be smooth, assume that $\cM(w)$ is invertible for every $w$ and that, for every $p>1$ and every $m \ge 0$, we have
\begin{equ}[e:boundintegr]
\E \bigl|\d_{k_1}\cdots \d_{k_m} X(w)\bigr|^p < \infty\;,\qquad 
\E \bigl\|\cM(w)^{-1}\bigr\|^p < \infty\;.
\end{equ}
Then the law of $X(w)$ has a smooth density with respect to Lebesgue measure. Furthermore, the derivatives of the law of 
$X$ can be bounded from above by expressions that depend only on the bounds \eref{e:boundintegr}, but are independent
of $N$, provided that $\sum \delta t_k = T$ remains fixed.
\end{theorem}

Besides Lemma~\ref{lem:IBP}, the main ingredient of the proof of Theorem~\ref{theo:smooth} is the following
 integration by parts formula which lies at the heart
of the success of Malliavin calculus.
If $F_k$ and $G$ are square integrable functions
with square integrable derivatives, then we have the identity
\begin{equs}
\E \Bigl(\int \cD_t G(w)\, F_t(w)\,dt\Bigr) &= \E \sum_k \d_k G(w) F_k(w)\,\delta t_k \\
&= \E G(w) \sum_k F_k(w)\, \delta w_k - \E  G(w) \sum_k \d_k F_k(w)\,\delta t_k \\
&\eqdef \E \Bigl(G(w) \int F_t\,dw(t)\Bigr)\;,\label{e:IBPsimple}
\end{equs}
where we \textit{defined} the Skorokhod integral  $\int F_t\,dw(t)$ by the expression on the second line.
Note that in order to obtain \eref{e:IBPsimple}, we only integrated by parts with respect to the variables
$\delta w_k$.

\begin{remark}
The Skorokhod integral is really an extension of the usual It\^o integral, which is the justification
for our notation. This is because, if $F_t$ is an adapted process, then $F_{t_k}$ is independent 
of $\delta w_\ell$ for $\ell \ge k$ by definition. As a consequence, the term $\d_k F_k$ drops and
we are reduced to the usual It\^o integral.
\end{remark}

\begin{remark}
It follows immediately from the definition that one has the identity
\begin{equ}[e:DInt]
\cD_t \int F_s\,dw(s) = F_t + \int \cD_t F_s\,dw(s)\;.
\end{equ}
Formally, one can think of this identity as being derived from the Leibnitz rule, combined with 
the identity $\cD_t (dw(s)) = \delta(t-s)\,ds$, which is a kind of continuous analogue of the trivial
discrete identity $\d_k \delta w_\ell = \delta_{k\ell}$.
\end{remark}

This Skorokhod integral satisfies the following extension of It\^o's isometry:

\begin{proposition}\label{prop:Ito}
Let $F_k$ be square integrable functions
with square integrable derivatives, then
\begin{equs}
\E \Bigl(\int F_t\,dw(t)\Bigr)^2 &= \E \int F_t^2(w)\, dt + \E \int \int \cD_t F_s(w)\,\cD_t F_s(w) \,ds\,dt\\
 &\le \E \int F_t^2(w)\, dt + \E \int \int |\cD_t F_s(w)|^2\,ds\,dt\;,
\end{equs}
holds.
\end{proposition}

\begin{proof}
It follows from the definition that one has the identity
\begin{equ}
\E \Bigl(\int F_t\,dw(t)\Bigr)^2 = \sum_{k,\ell}\E  \bigl(F_k F_\ell\,\delta w_k\delta w_\ell +  \d_k F_k\d_\ell F_\ell\,\delta t_k\delta t_\ell - 2 F_k\d_\ell F_\ell\,\delta w_k\delta t_\ell\bigr)\;.
\end{equ}
Applying the identity $\E G\,\delta w_\ell = \E \d_\ell G\,\delta t_\ell$ to
 the first term in the above formula (with $G = F_k F_\ell \,\delta w_k$), we thus obtain
\begin{equ}
\ldots = \sum_{k,\ell}\E  \bigl(F_k F_\ell \delta_{k,\ell} \delta t_\ell +  \d_k F_k\,\d_\ell F_\ell \,\delta t_k\delta t_\ell + (F_\ell\d_\ell F_k - F_k\d_\ell F_\ell)\,\delta w_k\delta t_\ell\bigr)\;.
\end{equ}
Applying the same identity to the last term then finally leads to 
\begin{equ}
\ldots = \sum_{k,\ell}\E  \bigl(F_k F_\ell \delta_{k,\ell} \delta t_\ell +  \d_k F_\ell \,\d_\ell F_k\,\delta t_k\delta t_\ell\bigr)\;,
\end{equ}
which is precisely the desired result.
\end{proof}

As a consequence, we have the following:
\begin{proposition}\label{prop:SkorLp}
Assume that $\sum \delta t_k = T < \infty$.
Then, for every $p > 0$ there exists $C>0$ and $k>0$ such that the bound
\begin{equ}
\E \Bigl|\int F_s\,dw(s) \Bigr|^p \le C \Bigl(1 + \sum_{0\le\ell \le k} \sup_{t_0,\ldots,t_\ell} \E \bigl|\cD_{t_1}\cdots \cD_{t_\ell} F_{t_0}\bigr|^{2p}\Bigr)\;,
\end{equ}
holds. Here, $C$ may depend on $T$ and $p$, but $k$ depends only on $p$.
\end{proposition}
\begin{proof}
Since the case $p \le 2$ follows from Proposition~\ref{prop:Ito}, we can assume 
without loss of generality that $p > 2$. Combining \eref{e:IBPsimple}
with \eref{e:DInt} and then applying H\"older's inequality, we have
\begin{equs}
\E \Bigl|\int &F_s\,dw(s) \Bigr|^p = (p-1)\,\E  \Bigl|\int F_s\,dw(s) \Bigr|^{p-2} \int F_t \Bigl(F_t + \int \cD_t F_s\,dw(s)\Bigr)\,dt \\
&\le {1\over 2} \E  \Bigl|\int F_s\,dw(s) \Bigr|^{p} + c\E \int  \Bigl|F_t + \int \cD_t F_s\,dw(s)\Bigr|^{2p \over 3}\,dt + c\E \int |F_t|^{2p}\,dt\\
&\le {1\over 2} \E  \Bigl|\int F_s\,dw(s) \Bigr|^{p} + c\E \int  \Bigl|\int \cD_t F_s\,dw(s)\Bigr|^{2p \over 3}\,dt + c\E \int (1+|F_t|)^{2p}\,dt\;.
\end{equs}
where $c$ is some constant depending on $p$ and $T$ that changes from line to line. The claim now follows by induction.
\end{proof}

\begin{remark}
The bound in Proposition~\ref{prop:SkorLp} is clearly very far from optimal. 
Actually, it is known that, for every $p \ge 1$, there exists $C$ such that
\begin{equ}
\E \Bigl|\int F_s\,dw(s) \Bigr|^{2p} \le C \E \Bigl|\int F_s^2\,ds \Bigr|^{p} + C \E \Bigl|\int |\cD_t F_s|^2\,ds\,dt \Bigr|^{p}\;,
\end{equ}
even if $T = \infty$.
However,
this extension of the Burkholder-Davies-Gundy inequality 
requires highly non-trivial harmonic analysis and, to best of the author's knowledge,
cannot be reduced to a short elementary calculation. The reader interested in knowing more
can find its proof in \cite[Ch.~1.3--1.5]{Nua95}.
\end{remark}

The proof of Theorem~\ref{theo:smooth} is now straightforward:

\begin{proof}[of Theorem~\ref{theo:smooth}]
We want to show that Lemma~\ref{lem:IBP} can be applied. 
For $\eta \in \R^n$, we then have from the definition of $\cM$ the identity
\begin{equ}[e:transDer]
\bigl(D_jG\bigr)(X(w)) = \sum_{k,m} \d_k \bigl(G(X(w))\bigr) \d_k X_m(w)\, \delta t_k\,  \cM^{-1}_{mj}(w)\;.
\end{equ}
Combining this identity with \eref{e:IBPsimple}, it follows that
\begin{equ}[e:MallDerG]
\E D_j G(X) = \E \Bigl(G(X(w)) \sum_{m} \int \cD_t X_m(w)\, \cM^{-1}_{mj}(w)\,dw(t)\Bigr)\;.
\end{equ}
Note that, by the chain rule, one has the identity
\begin{equ}
\cD_t \cM^{-1} = - \cM^{-1} \bigl(\cD_t \cM\bigr) \cM^{-1}\;,
\end{equ}
and similarly for higher order derivatives, so that the Malliavin 
derivatives of $\cM^{-1}$ can be bounded by terms involving $\cM^{-1}$
and the Malliavin derivatives of $X$.

Combining this with Proposition~\ref{prop:Ito} and \eref{e:boundintegr} immediately shows that the requested result holds
for $k=1$. Higher values of $k$ can be treated by induction by repeatedly applying \eref{e:transDer}. 
This will lead to expressions of the type \eref{e:MallDerG}, with the right hand side consisting of multiple Skorokhod
integrals of higher order polynomials in $\cM^{-1}$ and derivatives of $X$. 

By Proposition~\ref{prop:SkorLp},
the moments of each of the terms appearing in this way can be bounded by finitely many 
of the expressions appearing in the assumption
so that the required statement follows.
\end{proof}

\section{Application to Diffusion Processes}
\label{sec:Hormander}

We are now almost ready to tackle the proof of H\"ormander's theorem. Before we start,
we discuss how $\cD_s X_t$ can be computed when $X_t$ is the solution to an SDE of the
type \eref{e:SDE} and we use this discussion to formulate precise assumption for our theorem.

\subsection{Malliavin Calculus for Diffusion Processes}

By taking the limit $N \to \infty$ and $\delta t_k \to 0$ with $\sum \delta t_k = 1$, the
results in the previous section show that one can define a ``Malliavin derivative'' operator $\cD$, acting
on a suitable class of ``smooth'' 
random variables and returning a stochastic process that has all the usual properties of a derivative.
Let us see how it acts on the solution to an SDE of the type \eref{e:SDE}.

An important tool for our analysis will be the linearisation of \eref{e:SDE} with respect to its initial condition.
Denote by $\Phi_t$ the (random) solution map to \eref{e:SDE}, so that $x_t = \Phi_t(x_0)$.
It is then known that, under Assumption~\ref{ass:moments} below, $\Phi_t$ is almost surely a smooth map
for every $t$. We actually obtain a flow of smooth maps, namely a two-parameter family of maps $\Phi_{s,t}$
such that $x_t = \Phi_{s,t}(x_s)$ for every $s \le t$ and such that $\Phi_{t,u} \circ \Phi_{s,t} = \Phi_{s,u}$ and $\Phi_t = \Phi_{0,t}$.
For a given initial condition $x_0$, we then denote by $J_{s,t}$ the derivative of $\Phi_{s,t}$ evaluated at $x_s$.
Note that the chain rule immediately implies that one has the composition law $J_{s,u} = J_{t,u} J_{s,t}$, where the
product is given by simple matrix multiplication. We also use the notation $J_{s,t}^{(k)}$ for the $k$th-order derivative of $\Phi_{s,t}$.

It is straightforward to obtain an equation governing $J_{0,t}$ by
differentiating both sides of \eref{e:SDE} with respect to $x_0$. This yields the non-autonomous linear equation
\begin{equ}[e:Jac]
d J_{0,t} = DV_0(x_t)\,J_{0,t}\,dt + \sum_{i=1}^m DV_i(x_t)\,J_{0,t}\circ dW_i(t)\;,\qquad J_{0,0} = I\;,
\end{equ}
where $I$ is the $n\times n$ identity matrix.
Higher order derivatives $J_{0,t}^{(k)}$ with respect to the initial condition can be defined similarly.

\begin{remark}
For every $s>0$, the quantity $J_{s,t}$ solves the same equation as \eref{e:Jac}, except for the initial condition
which is given by $J_{s,s} = I$.
\end{remark}

On the other hand, we can use \eref{e:DInt} to, at least on a formal level,
take the Malliavin derivative of the integral form of \eref{e:SDE},
which then yields for $r \le t$ the identity
\begin{equ}
\cD_r^j X(t) = \int_r^t DV_0(X_s)\,\cD_r^j X_s\,ds + \sum_{i=1}^m\int_r^t DV_i(X_s)\,\cD_r^j X_s\circ dW_i(s) +   
V_j(X_r)\;.
\end{equ}
(Here we denote by $\cD^j$ the Malliavin derivative with respect to $W_j$; the generalisation of
the discussion of the previous section to the case of finitely many independent Wiener processes
is straightforward.)
We see that, save for the initial condition at time $t=r$ given by $V_j(X_r)$,  this equation is identical
to the integral form of \eref{e:Jac}! 

As a consequence, we have for $s < t$ the identity
\begin{equ}[e:MallDer]
\cD_s^j X_t = J_{s,t} V_j(X_s)\;.
\end{equ}
Furthermore, since $X_t$ is independent of the later increments of $W$, we have $\cD_s^j X_t = 0$ for $s \ge t$.

By the composition property $J_{0,t} = J_{s,t}J_{0,s}$, we can write $J_{s,t} = J_{0,t}J_{0,s}^{-1}$, which will be useful in the sequel.
Here, the inverse $J_{0,t}^{-1}$ of the Jacobian can be found by solving the SDE
\begin{equ}[e:JacInv]
d J_{0,t}^{-1} = -J_{0,t}^{-1}\,DV_0(x)\,dt - \sum_{i=1}^m J_{0,t}^{-1}\,DV_i(x)\circ dW_i\;.
\end{equ}
This follows from the chain rule
by noting that if we denote by $\Psi(A) = A^{-1}$ the map that takes the inverse of a
square matrix, then we have $D\Psi(A) H = -A^{-1} H A^{-1}$.

This discussion is the motivation for the following assumption, which we assume to be in
force from now on:
\begin{assumption}\label{ass:moments}
The vector fields $V_i$ are $\CC^\infty$ and all of their derivatives grow at most polynomially at infinity. 
Furthermore, they are such that the solutions to \eref{e:SDE}, \eref{e:Jac} and \eref{e:JacInv} satisfy
\begin{equ}
\E \sup_{t \le T} |x_t|^p < \infty \;,\qquad \E \sup_{t \le T} |J_{0,t}^{(k)}|^p < \infty
\;,\qquad \E \sup_{t \le T} |J_{0,t}^{-1}|^p < \infty\;,
\end{equ}
for every initial condition $x_0 \in \R^n$, every terminal time $T>0$, every $k>0$,
and every $p > 0$.
\end{assumption}

\begin{remark}
It is well-known that Assumption~\ref{ass:moments} holds if the $V_i$ are bounded with
bounded derivatives of all orders. However, this is far from being a necessary assumption.
\end{remark}

\begin{remark}
Under Assumption~\ref{ass:moments}, standard limiting procedures allow to justify 
\eref{e:MallDer}, as well as all the formal manipulations that we will perform in the sequel.
\end{remark}

With these assumptions in place, the version of H\"ormander's theorem that
we are going to prove in these notes is as follows:

\begin{theorem}\label{theo:main}
Let $x_0 \in \R^n$ and let $x_t$ be the solution to \eref{e:SDE}. If the vector fields $\{V_j\}$ satisfy the
parabolic H\"ormander condition and Assumption~\ref{ass:moments} is satisfied, then the law
of $X_t$ has a smooth density with respect to Lebesgue measure.
\end{theorem}

\begin{proof}
Denote by $\cA_{0,t}$ the operator $\cA_{0,t}v =  \int_0^t J_{s,t} V(X_s) v(s)\,ds$, 
where $v$ is a square integrable, not
necessarily adapted, $\R^m$-valued stochastic  process and $V$ is the 
$n \times m$ matrix-valued function obtained
by concatenating the vector fields $V_j$ for $j=1,\ldots,m$.
With this notation, it follows from \eref{e:MallDer} that the Malliavin covariance matrix $\cM_{0,t}$ of $X_t$ 
is given by
\begin{equ}
\cM_{0,t} = \cA_{0,t} \cA_{0,t}^* = \int_0^t J_{s,t} V(X_s)V^*(X_s) J_{s,t}^*\,ds\;.
\end{equ}
It follows from \eref{e:MallDer} that the assumptions of Theorem~\ref{theo:smooth} are satisfied
for the random variable $X_t$, provided that we can show that $\|\cM_{0,t}^{-1}\|$ has bounded
moments of all orders. This in turn follows by combining Lemma~\ref{lem:invM} with Theorem~\ref{theo:Hor2}
below.  
\end{proof}

\subsection{Proof of H\"ormander's Theorem}

The remainder of this section is devoted to a proof of the fact that 
H\"ormander's condition is sufficient to guarantee the invertibility of the Malliavin matrix of a diffusion process.
For purely technical reasons, it turns out to be advantageous to rewrite the Malliavin matrix as
\begin{equ}
\cM_{0,t} = J_{0,t} \cC_{0,t} J_{0,t}^*\;,\qquad  \cC_{0,t} = \int_0^t J_{0,s}^{-1} V(X_s)V^*(X_s) \bigl(J_{0,s}^{-1}\bigr)^*\,ds\;,
\end{equ}
where $\cC_{0,t}$ is the \textit{reduced Malliavin matrix} of our diffusion process. 

\begin{remark}
The reason for considering the reduced Malliavin matrix is that the process appearing under the 
integral in the definition of $\cC_{0,t}$ is adapted to the filtration generated by $W_t$. This allows us to
use some tools from stochastic calculus that would not be available otherwise.
\end{remark}

Since we assumed that $J_{0,t}$
has inverse moments of all orders, the invertibility of $\cM_{0,t}$ is equivalent to that of $\cC_{0,t}$. Note first that 
since $\cC_{0,t}$ is a positive definite symmetric matrix, the norm of its inverse is given by
\begin{equ}
\|\cC_{0,t}^{-1}\| = \Bigl(\inf_{|\eta| = 1} \scal{\eta, \cC_{0,t}\eta}\Bigr)^{-1}\;.
\end{equ}

A very useful observation is then the following:

\begin{lemma}\label{lem:invM}
Let $M$ be a symmetric positive semidefinite $n\times n$ matrix-valued random variable such that $\E \|M\|^p < \infty$ for every $p\ge 1$ and such that, 
for every $p \ge 1$ there
exists $C_p$ such that 
\begin{equ}[e:assM]
\sup_{|\eta| = 1} \P \bigl(\scal{\eta,M\eta} <\eps\bigr) \le C_p \eps^p\;,
\end{equ}
holds for every $\eps \le 1$. Then, $\E \|M^{-1}\|^p < \infty$ for every $p \ge 1$.
\end{lemma}

\begin{proof}
The non-trivial part of the result is that the supremum over $\eta$ is taken outside of the probability in \eref{e:assM}.
For $\eps > 0$, let $\{\eta_k\}_{k \le N}$ be a sequence of vectors with $|\eta_k| = 1$ such that for every $\eta$ with $|\eta| \le 1$, 
there exists $k$ such that $|\eta_k - \eta| \le \eps^2$. It is clear that one can find such a set with $N \le C\eps^{2-2n}$
for some $C>0$ independent of $\eps$.
We then have the bound
\begin{equs}
\scal{\eta,M\eta} &= \scal{\eta_k,M\eta_k} + \scal{\eta-\eta_k,M\eta} + \scal{\eta-\eta_k,M\eta_k} \\
& \ge \scal{\eta_k,M\eta_k} - 2\|M\| \eps^2\;,
\end{equs}
so that
\begin{equs}
\P\Bigl(\inf_{|\eta| = 1} \scal{\eta,M\eta} \le \eps\Bigr)&\le 
\P\Bigl(\inf_{k \le N} \scal{\eta_k,M\eta_k} \le 4\eps\Bigr) + \P\Bigl(\|M\| \ge {1\over \eps}\Bigr) \\
&\le C \eps^{2-2n} \sup_{|\eta|=1} \P\Bigl(\scal{\eta,M\eta} \le 4\eps\Bigr) + \P\Bigl(\|M\| \ge {1\over \eps}\Bigr)\;.
\end{equs}
It now suffices to use \eref{e:assM} for $p$ large enough to bound the first term and Chebychev's inequality
combined with the moment bound on $\|M\|$ to bound the second term.
\end{proof}

As a consequence of this, Theorem~\ref{theo:main} is a corollary of:

\begin{theorem}\label{theo:Hor2}
Consider \eref{e:SDE} and assume that Assumption~\ref{ass:moments} holds. If the corresponding vector fields satisfy the parabolic 
H\"ormander condition then, for every initial condition $x \in \R^n$, we have the bound
\begin{equ}
\sup_{|\eta| = 1} \P \bigl(\scal{\eta,\cC_{0,1} \eta} <\eps\bigr) \le C_p \eps^p\;,
\end{equ}
for suitable constants $C_p$ and all $p \ge 1$.
\end{theorem}

\begin{remark}
The choice $t=1$ as the final time is of course completely arbitrary. Here and in the sequel, we will 
always consider functions on the time interval $[0,1]$.
\end{remark}

Before we turn to the proof of this result, we introduce a very
useful notation which, to the best of the author's knowledge, was first used in \cite{ErgodicBig}. 
Given a family $A = \{A_\eps\}_{\eps \in (0,1]}$ of events depending
on some parameter $\eps > 0$, we say that $A$ is ``almost true'' if, for every $p>0$ there exists a constant $C_p$ such that 
$\P(A_\eps) \ge 1-C_p \eps^p$ for all $\eps \in (0,1]$. Similarly for ``almost false''.
Given two such families of events $A$ and $B$, we say that ``$A$ almost implies $B$'' and we write $A \Rightarrow_\eps B$
if $A \setminus B$ is almost false. It is straightforward to check that these notions behave as expected (almost implication is
transitive, finite unions of almost false events are almost false, etc). Note also that these notions are unchanged
under any reparametrisation of the form $\eps \mapsto \eps^\alpha$ for $\alpha > 0$.
Given two families $X$ and $Y$ of real-valued random variables, we will similarly write $X \le_\eps Y$ as a shorthand
for the fact that $\{X_\eps \le Y_\eps\}$ is ``almost true''.

Before we proceed, we state the following useful result, where $\|\cdot\|_\infty$ denotes the $L^\infty$ norm
and $\|\cdot\|_\alpha$ denotes the best possible $\alpha$-H\"older constant.

\begin{lemma}\label{lem:dtf}
  Let $f \colon [0,1] \to \R$ be continuously differentiable and let
  $\alpha \in (0,1]$. Then, the bound
  \begin{equ}
   \|\d_t f\|_{\infty} = \| f\|_{1} \le 4\|f\|_{\infty} \max \Bigl\{ 1 , \|f\|_{\infty}^{-{1 \over 1+\alpha}} \|
   \d_t  f\|_{\alpha}^{1\over 1+\alpha} \Bigr\}
  \end{equ}
  holds, where $\|f\|_\alpha$ denotes the best $\alpha$-H\"older constant for $f$.
\end{lemma}

\begin{proof}
  Denote by $x_0$ a point such that $|\d_t f(x_0)| = \|\d_t
  f\|_{\infty}$. It follows from the definition of the
  $\alpha$-H\"older constant $\|\d_t f\|_{\CC^\alpha}$ that $|\d_t
  f(x)| \ge {1\over 2} \|\d_t f\|_{\infty}$ for every $x$ such that
  $|x - x_0| \le \bigl(\|\d_t f\|_{\infty} / 2 \|\d_t
  f\|_{\CC^\alpha}\bigr)^{1/\alpha}$.  The claim then follows from the
  fact that if $f$ is continuously differentiable and $|\d_t f(x)| \ge A$ over an interval $I$, then there
  exists a point $x_1$ in the interval such that $|f(x_1)| \ge
  A|I|/2$.
\end{proof}

With these notations at hand, we have the following statement, which is essentially a quantitative
version of the Doob-Meyer decomposition theorem. Originally, it appeared in \cite{Nor86SMC},
although some form of it was already present in earlier works. The statement and proof given here are slightly different
from those in \cite{Nor86SMC}, but are very close to them in spirit.

\begin{lemma}
Let $W$ be an $m$-dimensional Wiener process and let $A$ and $B$ be $\R$ and $\R^m$-valued 
adapted processes such that, for $\alpha = {1\over 3}$, one has $\E \bigl(\|A\|_\alpha + \|B\|_\alpha\bigr)^p < \infty$ for every $p$.
Let $Z$ be the process defined by
\begin{equ}[e:defZ]
Z_t = Z_0 + \int_0^t A_s\,ds + \int_0^t B_s\,dW(s)\;.
\end{equ}
Then, there exists a universal constant $r \in (0,1)$ such that one has
\begin{equ}
\bigl\{\|Z\|_\infty < \eps\bigr\} \quad\Rightarrow_\eps\quad \bigl\{\|A\|_\infty < \eps^r\bigr\}\;\&\;\bigl\{\|B\|_\infty < \eps^r\bigr\}\;.
\end{equ}
\end{lemma}

\begin{proof}
Recall the exponential martingale inequality \cite[p.~153]{Yor}, stating that if $M$ is any continuous martingale with quadratic
variation process $\scal{M}(t)$, then
\begin{equ}
\P \Bigl(\sup_{t \le T} |M(t)| \ge x\quad\&\quad \scal{M}(T) \le y\Bigr) \le 2\exp \bigl(-x^2/2y\bigr)\;,
\end{equ}
for every positive $T$, $x$, $y$.
With our notations, this implies that for any $q < 1$ and any adapted process $F$, 
one has the almost implication
\begin{equ}[e:expMart]
\bigl\{\|F\|_\infty < \eps\bigr\} \quad\Rightarrow_\eps\quad \Bigl\{\Bigl\|\int_0^\cdot F_t\,dW(t)\Bigr\|_\infty < \eps^q\Bigr\}\;. 
\end{equ}
With this bound in mind, we apply It\^o's formula to $Z^2$, so that
\begin{equ}[e:Z2]
Z_t^2 = Z_0^2 + 2\int_0^t Z_s\,A_s\,ds + 2\int_0^t Z_s\,B_s\,dW(s) + \int_0^t B_s^2\,ds\;.
\end{equ}
Since $\|A\|_\infty \le_\eps \eps^{-1/4}$ (or any other negative exponent for that matter) 
by assumption and similarly for $B$, it follows from this and \eref{e:expMart} that 
\begin{equ}
\bigl\{\|Z\|_\infty < \eps\bigr\} \quad\Rightarrow_\eps\quad \Bigl\{\Bigl|\int_0^1 A_s\,Z_s\,ds\Bigr| \le \eps^{3\over 4}\Bigr\}\
\;\&\; \Bigl\{\Bigl|\int_0^1 B_s\,Z_s\,dW(s)\Bigr| \le \eps^{2\over 3}\Bigr\}\;.
\end{equ}
Inserting these bounds back into \eref{e:Z2} and applying Jensen's inequality then yields
\begin{equ}
\bigl\{\|Z\|_\infty < \eps\bigr\} \quad\Rightarrow_\eps\quad \Bigl\{\int_0^1 B_s^2\,ds \le \eps^{1\over 2}\Bigr\}\
\quad\Rightarrow\quad \Bigl\{\int_0^1 |B_s|\,ds \le \eps^{1\over 4}\Bigr\}\;.
\end{equ}
We now use the fact that $\|B\|_\alpha \le_\eps \eps^{-q}$ for 
every $q > 0$ and we apply Lemma~\ref{lem:dtf} with $\d_t f(t) = |B_t|$ (we actually do it component by component), so that 
\begin{equ}
\bigl\{\|Z\|_\infty < \eps\bigr\} \quad\Rightarrow_\eps\quad  \bigl\{\|B\|_\infty \le \eps^{1\over 17}\bigr\}\;,
\end{equ}
say. In order to get the bound on $A$, note that we can again apply the exponential martingale inequality
to obtain that this ``almost implies'' the martingale part in \eref{e:defZ} is ``almost bounded'' in the supremum norm by $\eps^{1\over 18}$,
so that 
\begin{equ}
\bigl\{\|Z\|_\infty < \eps\bigr\} \quad\Rightarrow_\eps\quad  \Bigl\{\Bigl\|\int_0^\cdot A_s\,ds\Bigr\|_\infty \le \eps^{1\over 18}\Bigr\}\;.
\end{equ}
Finally  applying again Lemma~\ref{lem:dtf} with $\d_t f(t) = A_t$, we obtain that 
\begin{equ}
\bigl\{\|Z\|_\infty < \eps\bigr\} \quad\Rightarrow_\eps\quad  \bigl\{\|A\|_\infty \le \eps^{1/80}\bigr\}\;,
\end{equ}
and the claim follows with $r = 1/80$.
\end{proof}

\begin{remark}
By making $\alpha$ arbitrarily close to $1/2$, keeping track of the different norms appearing
in the above argument, and then bootstrapping the argument, it is possible to show that 
\begin{equ}
\bigl\{\|Z\|_\infty < \eps\bigr\} \quad\Rightarrow_\eps\quad  \bigl\{\|A\|_\infty \le \eps^{p}\bigr\}\;\&\; \bigl\{\|B\|_\infty \le \eps^{q}\bigr\}\;,
\end{equ}
for $p$ arbitrarily close to $1/5$ and $q$ arbitrarily close to $3/10$. This seems to be a very small improvement over the exponent
$1/8$ that was originally obtained in \cite{Nor86SMC}, but is certainly not optimal either. 
The main reason why our result is suboptimal is that we move several times back and 
forth between $L^1$, $L^2$, and $L^\infty$ norms.
(Note furthermore that our result is not really comparable
to that in \cite{Nor86SMC}, since Norris used $L^2$ norms in the statements and his assumptions were slightly different from ours.)
\end{remark}

We now have all the necessary tools to prove Theorem~\ref{theo:Hor2}:

\begin{proof}[of Theorem~\ref{theo:Hor2}]
We fix some initial condition $x_0 \in \R^n$ and some unit vector $\eta \in \R^n$. With the notation
introduced earlier, our aim is then to show that
\begin{equ}[e:startingPoint]
\bigl\{\scal{\eta,\cC_{0,1}\eta} < \eps\bigr\} \quad\Rightarrow_\eps\quad \emptyset\;,
\end{equ}
or in other words that the statement $\scal{\eta,\cC_{0,1}\eta} < \eps$ is ``almost false''.
As a shorthand, we introduce for an arbitrary smooth vector field $F$ on $\R^n$ the process $Z_F$ defined by
\begin{equ}
Z_F(t) = \scal{\eta, J_{0,t}^{-1} F(x_t)}\;,
\end{equ}
so that 
\begin{equ}[e:boundC]
\scal{\eta,\cC_{0,1}\eta} = \sum_{k=1}^m \int_0^1 |Z_{V_k}(t)|^2\,dt \ge  \sum_{k=1}^m \Bigl(\int_0^1 |Z_{V_k}(t)|\,dt\Bigr)^2\;.
\end{equ}
The processes $Z_F$ have the nice property that they solve the stochastic differential equation
\begin{equ}[e:ZF]
d Z_F(t) = Z_{[F,V_0]}(t)\,dt + \sum_{i=1}^m Z_{[F,V_k]}(t)\circ dW_k(t)\;,
\end{equ}
which can be rewritten in It\^o form as
\begin{equ}[e:ItoZF]
d Z_F(t) = \Bigl(Z_{[F,V_0]}(t) + \sum_{k=1}^m{1\over 2}Z_{[[F,V_k], V_k]}(t)\Bigr)\,dt + \sum_{i=1}^m Z_{[F,V_k]}(t)\, dW_k(t)\;.
\end{equ}
Since we assumed that all derivatives of the $V_j$ grow at most polynomially, we deduce from
the H\"older regularity of Brownian motion that, provided that the derivatives of $F$ grow at most polynomially fast,
$Z_F$ does indeed satisfy the assumptions on 
its H\"older norm required for the application of Norris's lemma. 
The idea now is to observe that, by \eref{e:boundC}, the left hand side of \eref{e:startingPoint} 
states that $Z_F$ is ``small'' for every $F \in \cV_0$.
One then argues that, by Norris's lemma, if $Z_{F}$ is small for every $F \in \cV_k$
then, by considering \eref{e:ZF}, it follows that $Z_F$ is also small for every $F \in \cV_{k+1}$.
H\"ormander's condition then ensures that a contradiction arises at some stage, since 
 $Z_F(0) = \scal{F(x_0), \xi}$ and there exists $k$ such that $\cV_k(x_0)$ spans all of $\R^n$.

Let us make this rigorous. It follows from Norris's lemma and \eref{e:ItoZF} that one has the almost implication
\begin{equ}
\bigl\{\|Z_F\|_\infty < \eps\bigr\} \quad \Rightarrow_\eps \quad 
\bigl\{\|Z_{[F,V_k]}\|_\infty < \eps^r\bigr\} \;\&\; \bigl\{\|Z_G\|_\infty < \eps^r\bigr\} \;,
\end{equ}
for $k=1,\ldots,m$ and for $G = [F,V_0] + {1\over 2}\sum_{k=1}^m [[F,V_k], V_k]$.
Iterating this bound a second time, this time considering the equation for $Z_G$,
we obtain that 
\begin{equ}
\bigl\{\|Z_F\|_\infty < \eps\bigr\} \quad \Rightarrow_\eps \quad 
\bigl\{\|Z_{[[F,V_k], V_\ell]}\|_\infty < \eps^{r^2}\bigr\} \;,
\end{equ}
so that we finally obtain the implication
\begin{equ}[e:recurrence]
\bigl\{\|Z_F\|_\infty < \eps\bigr\} \quad \Rightarrow_\eps \quad 
\bigl\{\|Z_{[F,V_k]}\|_\infty < \eps^{r^2}\bigr\} \;,
\end{equ}
for $k=0,\ldots,m$.

At this stage, we are basically done. Indeed, combining \eref{e:boundC} with Lemma~\ref{lem:dtf} as above,
we see that
\begin{equ}
\bigl\{\scal{\eta,\cC_{0,1}\eta} < \eps \bigr\} \quad\Rightarrow_\eps\quad
\bigl\{\|Z_{V_k}\|_\infty < \eps^{1/5}\bigr\}\;.
\end{equ}
Applying \eref{e:recurrence} iteratively, we see that for every $k>0$ there exists some $q_k>0$ such that 
\begin{equ}
\bigl\{\scal{\eta,\cC_{0,1}\eta} < \eps \bigr\} \quad\Rightarrow_\eps\quad
\bigcap_{V \in \cV_k} \bigl\{\|Z_{V}\|_\infty < \eps^{q_k}\bigr\}\;.
\end{equ}
Since $Z_V(0) = \scal{\eta, V(x_0)}$ and since there exists some $k>0$ such that $\cV_k(x_0) = \R^n$, the right hand side
of this expression is empty for some sufficiently large value of $k$, which is precisely the desired result.
\end{proof}

\bibliographystyle{Martin}
\bibliography{./refs}

\end{document}